\documentclass[11pt]{article}

\usepackage{amsmath, amsfonts, amsthm, verbatim}

\newcommand{\R}{\mathbb{R}}
\newcommand{\N}{\mathbb{N}}
\newcommand{\Z}{\mathbb{Z}}
\newcommand{\E}{\mathbf{E}}
\newcommand{\p}{\mathbf{P}}
\newcommand{\dint}{\int \!\! \int}

\theoremstyle{plain}
\newtheorem{lemma}{Lemma}
\newtheorem{theorem}{Theorem}

\newtheorem{corollary}{Corollary}

\begin{document}

\title{A note on the Kesten--Grincevi\v{c}ius--Goldie theorem}

\author{P\'eter Kevei\thanks{
Center for Mathematical Sciences, Technische Universit\"at M\"unchen, 
Boltzmannstra{\ss }e 3, 85748 Garching, Germany,
\texttt{peter.kevei@tum.de}}}

\date{}

\maketitle

\begin{abstract}
\noindent
Consider the perpetuity equation $X \stackrel{\mathcal{D}}{=} A X + B$,
where $(A,B)$ and $X$ on the right-hand side are independent. 
The Kesten--Grincevi\v{c}ius--Goldie theorem states that $\p \{ X > x \} \sim c 
x^{-\kappa}$ if
$\E A^\kappa = 1$, $\E A^\kappa \log_+ A < \infty$, and $\E |B|^\kappa < \infty$.
We assume that $\E |B|^\nu < \infty$ for some $\nu > \kappa$, 
and consider two cases (i) $\E A^\kappa = 1$, 
$\E A^\kappa \log_+ A = \infty$; (ii) $\E A^\kappa < 1$, $\E A^t = \infty$ 
for all $t > \kappa$.
We show that under appropriate additional assumptions on 
$A$ the asymptotic $ \p \{ X > x \}  \sim c x^{-\kappa} \ell(x) $ holds,
where $\ell$ is a nonconstant slowly varying function. We use Goldie's renewal theoretic 
approach.
\end{abstract}

\section{Introduction and results}

Consider the perpetuity equation
\begin{equation} \label{eq:equation}
X \stackrel{\mathcal{D}}{=} A X + B,
\end{equation}
where $(A,B)$ and $X$ on the right-hand side are independent. 
To exclude degenerate cases as usual we assume that $\p \{ A x + B = x\} < 1$ for any
$x \in \R$. We also assume that $A \geq 0$, $A \not\equiv 1$, and that $\log A$ 
conditioned on $A \neq 0$ is nonarithmetic.

The first results on existence and tail behavior of the solution are due to
Kesten \cite{Kesten}, who proved that if
\begin{equation} \label{eq:Goldie-cond}
\begin{gathered}
\E A^\kappa = 1, \ \E A^\kappa \log_+ A < \infty,
\ \log A \text{ conditioned on $A \neq 0$ is nonarithmetic,} \\
\text{and } \E |B|^\kappa < \infty \text{ for some  } \kappa > 0,
\end{gathered}
\end{equation}
where $\log_+ x = \max\{ \log x, 0\}$, then 
the solution of (\ref{eq:equation}) has 
Pareto-like tail, i.e.
\begin{equation} \label{eq:KG}
\p \{ X > x \} \sim c_+ x^{-\kappa} \text{ and } 
\p \{ X < -x \} \sim c_{-} x^{-\kappa} \text{ as } x \to \infty
\end{equation}
for some $c_+, c_{-} \geq 0, c_+ + c_{-} > 0$. 
(In the following any nonspecified limit relation is meant as $x \to \infty$.)
Actually, Kesten proved a similar statement in $d$ dimension.
Later Goldie \cite{Goldie} simplified the proof of the same  result in the 
one-dimensional case (for more general equations) using renewal 
theoretic methods. His method is based on ideas from Grincevi\v{c}ius \cite{Grinc}, who 
partly rediscovered Kesten's results.
We refer to the implication (\ref{eq:Goldie-cond}) $\Rightarrow$ (\ref{eq:KG}) as the 
Kesten--Grincevi\v{c}ius--Goldie theorem.
That is, under general conditions on $A$, if $\p \{ A > 1 \} > 0$ the tail decreases 
at least polynomially. In this direction, we mention a result by 
Dyszewski \cite{Dys}, who showed that the tail of the solution of (\ref{eq:equation}) can 
even be slowly varying. On the other hand, Goldie and 
Gr\"ubel \cite{GolGru} showed that the solution has at least exponential tail under the 
assumption $A \leq 1$ a.s. For further results in the thin-tailed case see
Hitczenko and Weso{\l}owski \cite{HW}. Returning to the heavy-tailed case Grey 
\cite{Grey} showed that if 
$\E A^{\kappa} <1$, $\E A^{\kappa + \epsilon} < \infty$, then the tail of $X$ is 
regularly 
varying with parameter $-\kappa$ if and only if the tail of $B$ is.
Grey's results are also based on the previous results by Grincevi\v{c}ius \cite{Grinc}.

That is, the regular variation of the solution $X$ of (\ref{eq:equation}) is 
either caused by $A$ alone, or by $B$ alone (under some weak condition on the 
other variable). Our intention in the present note is to explore more the role 
of $A$, i.e.~to extend the Kesten--Grincevi\v{c}ius--Goldie theorem.
More precisely, we assume that $\E |B|^\nu < \infty$ 
for some $\nu > \kappa$, and we obtain sufficient conditions on $A$ that 
imply $\p \{  X > x \} \sim \ell(x) x^{-\kappa}$, where $\ell(\cdot)$ is some 
nonconstant slowly varying function.

The perpetuity 
equation (\ref{eq:equation}) has a wide range of applications; we only mention the ARCH 
and GARCH 
models in financial time series analysis, see Embrechts, Kl\"uppelberg and 
Mikosch 
\cite[Section 8.4 Perpetuities and ARCH Processes]{EKM}. For a complete account 
on the equation (\ref{eq:equation}) refer 
to Buraczewski, Damek and Mikosch \cite{BDM}. Equation (\ref{eq:equation}) is also 
strongly related to exponential functional of L\'evy processes, see Arista and Rivero
\cite{AR} and Behme and Lindner \cite{BL} and the references therein.

The key idea in Goldie's proof is to introduce the new probability 
measure
\begin{equation} \label{eq:alphameasure}
\p_\kappa \{ \log A \in C \} = \E [I(\log A \in C) A^\kappa],
\end{equation}
where $I(\cdot)$ stands for the indicator function.
Since $\E A^\kappa = 1$ this is indeed a probability measure.
If $F$ is the distribution function (df) of $\log A$ under $\p$, then under $\p_\kappa$
\begin{equation} \label{eq:Fkappa}
F_\kappa(x) = \p_\kappa \{ \log A \leq x \} =
\int_{-\infty}^x e^{\kappa y} F({\mathrm{d}} y).
\end{equation}
Under $\p_\kappa$ equation (\ref{eq:equation}) can be rewritten as a 
renewal 
equation, where the renewal function corresponds to $F_\kappa$.
If $\E_\kappa \log A = \E A^\kappa \log A \in (0, \infty)$, then
a renewal theorem on the line
implies the required tail asymptotics. Yet a smoothing 
transformation and a Tauberian argument is needed, since key renewal theorems 
apply only for direct Riemann integrable functions.

What we assume instead of the finiteness of the mean is that under $\p_\kappa$ 
the variable $\log A$ is in the domain of attraction of a stable law with index
$\alpha \in (0,1]$, i.e.~$\log A \in D(\alpha)$. Since
\begin{equation} \label{eq:Fknegtail}
F_\kappa(-x) = \p_\kappa \{ \log A \leq -x \} = \E I(\log A \leq -x) A^\kappa 
\leq 
e^{-\kappa x},
\end{equation}
under $\p_\kappa$ the variable $\log A$ belongs to $D(\alpha)$ if and only if
\begin{equation} \label{eq:assump-2}
1 - F_\kappa(x) = \overline F_\kappa(x) = \frac{\ell(x)}{x^\alpha}, 
\end{equation}
where $\ell$ is a slowly varying function.
Let $U(x) = \sum_{n=0}^\infty F_\kappa^{*n}(x)$ denote the renewal function of $\log A$ 
under $\p_\kappa$. Note that $U(x) < \infty$ for all $x \in \R$, since the random walk
$(S_n = \log A_1 + \ldots + \log A_n)_{n \geq 1}$ drifts to infinity under $\p_\kappa$ 
and $\E_\kappa [(\log A)_-]^2 < \infty$ by (\ref{eq:Fknegtail}); see 
Theorem 2.1 by Kesten and Maller \cite{KestenMaller}. Put 
\[ 
m(x) = \int_0^x [ F_\kappa(-u) +\overline F_\kappa (u)] {\mathrm{d}} u
\sim \int_0^x \overline F_\kappa(u) {\mathrm{d}} u
\sim \frac{\ell(x) x^{1-\alpha}}{1-\alpha}
\] 
for the truncated expectation; the first asymptotic follows from (\ref{eq:Fknegtail}),
the second from (\ref{eq:assump-2}), and holds only for $\alpha \neq 1$.
To obtain the asymptotic behavior of the solution of the renewal equation we have to use 
a key renewal theorem for random variables with infinite mean.
The infinite mean analogue of the strong renewal theorem (SRT) is the convergence
\begin{equation} \label{eq:renewal-alpha}
\lim_{x \to \infty} m(x) [U(x+h) - U(x) ] = h C_\alpha, \quad \forall h > 0,
\quad \text{where } \ C_\alpha = [\Gamma(\alpha) \Gamma(2-\alpha)]^{-1}.
\end{equation}
The first infinite mean SRT was shown by Garsia and Lamperti \cite{GarLam} in 1963 for 
nonnegative integer valued random variables, which was extended to the nonarithmetic case 
by Erickson \cite{Eri70, Eri71}. In both cases it was shown that for
$\alpha \in (1/2, 1]$ (in \cite{GarLam} $\alpha < 1$)
assumption (\ref{eq:assump-2}) implies the SRT, while for $\alpha 
\leq 1/2$ further assumptions are needed. For $\alpha \leq 1/2$ sufficient conditions for 
(\ref{eq:renewal-alpha}) were given by Chi \cite{Chi}, Doney \cite{Doney2}, Vatutin and 
Topchii \cite{VT}. The necessary and sufficient condition for nonnegative random 
variables was given independently by Caravenna \cite{Car} and Doney \cite{Doney}. They 
showed that if for a nonnegative random variable with 
df $H$ (\ref{eq:assump-2}) holds with
$\alpha \leq 1/2$, then (\ref{eq:renewal-alpha}) holds if and 
only if
\begin{equation} \label{eq:Doney-cond}
\lim_{\delta \to 0} \limsup_{x \to \infty} x \overline H(x) 
\int_1^{\delta x} \frac{1}{y \overline H(y)^2} H(x - {\mathrm{d}} y) = 0.
\end{equation}
We need this result in our case, where the random variable is not 
necessarily 
positive, but the left tail is exponential. This follows along the same lines as the 
proof of the SRT in \cite{Car}. For the sake of completeness, in the 
Appendix we 
sketch the proof of this result, see Theorem \ref{th:SRT}. For further results and 
history about the infinite mean 
SRT we refer 
to \cite{Car, Doney} and the references therein. In Lemma \ref{lemma:keyren} below,
which is a modification of Erickson's Theorem 3 \cite{Eri70}, we prove the corresponding 
key renewal theorem. Since in the literature (\cite[Lemma 3]{Rivero},
\cite[Theorem 4]{VT}) this lemma is stated incorrectly, we give a counterexample in the 
Appendix.
We use the notation $x_+ = \max \{ x, 0 \}$, $x_- = \max \{ -x, 0 \}$,
$x \in \R$.
Summarizing, our assumptions on $A$ are the following:
\begin{equation} \label{eq:A-ass}
\begin{gathered}
\E A^\kappa = 1, \ \text{(\ref{eq:assump-2}) and (\ref{eq:Doney-cond}) holds for 
$F_\kappa$ for some $\kappa > 0$ and $\alpha \in (0,1]$,} \\
\text{and  $\log A$ conditioned on $A \neq 0$ is nonarithmetic}.
\end{gathered}
\end{equation}

\begin{theorem} \label{th:=1}
Assume (\ref{eq:A-ass}), $\E |B|^\nu < \infty$ for some $\nu > \kappa$, 
and $\p \{ A x + B = x\} < 1$ for any $x \in \R$. 
Then for the tail of the solution of the perpetuity equation 
(\ref{eq:equation}) we have
\begin{equation} \label{eq:tail}
\begin{split}
& \lim_{x \to \infty} m(\log x) x^\kappa \p \{ X > x \}  = C_\alpha
\frac{1}{\kappa} \E [ (AX + B)_+^\kappa - (AX)_+^\kappa ],\\
& \lim_{x \to \infty} m(\log x) x^\kappa \p \{ X \leq -x \}  = C_\alpha
\frac{1}{\kappa} \E [ (AX + B)_-^\kappa - (AX)_-^\kappa ].
\end{split}
\end{equation}
Moreover,
$\E [ (AX + B)_+^\kappa - (AX)_+^\kappa ] +
\E [ (AX + B)_-^\kappa - (AX)_-^\kappa ] > 0$.
\end{theorem}

Theorem \ref{th:=1} is stated as a conjecture/open problem in
\cite[Problem 1.4.2]{Iksanov} by Iksanov.

The assumption $\p \{ A x + B = x\} < 1$ is only needed to ensure 
$\E [ (AX + B)_+^\kappa - (AX)_+^\kappa ] +\E [ (AX + B)_-^\kappa - (AX)_-^\kappa ] > 0$.
Indeed, if $A x_0 + B \equiv x_0$ for some $x_0 \in \R$, then the solution of
(\ref{eq:equation}) is $X \equiv x_0$, for which (\ref{eq:tail}) trivially holds, with 
both constants on the right equal to 0.

The conditions of the theorem are stated in terms of the properties of $A$ under the new 
measure $\p_\kappa$. Simple properties of regularly varying functions imply that if
$e^{\kappa x} \overline F(x) = \alpha \, \ell(x) /(\kappa \, x^{\alpha +1})$
with a slowly varying function $\ell$, then (\ref{eq:assump-2}) holds. 
See the remark after Theorem 2 \cite{Kor} by Korshunov.

Using the same methods, Goldie obtained tail asymptotics for solutions of 
more general random equations. The extension of these 
results to our setup is straightforward. We mention a particular example, because 
in the proof of the positivity of the constant in Theorems \ref{th:=1} and  \ref{th:<1} 
we need a result on the maximum of a random walk.

Consider the equation
\begin{equation} \label{eq:maxeq}
X \stackrel{\mathcal{D}}{=} AX \vee B,
\end{equation}
where $a \vee b = \max \{ a,b \}$, $A \geq 0$ and $(A,B)$ and $X$ on the right-hand 
side are independent. Theorem 5.2 in \cite{Goldie} states that if
(\ref{eq:Goldie-cond}) holds, then there is a unique solution $X$ to (\ref{eq:maxeq}), 
and 
$\p \{ X > x \} \sim c x^{-\kappa}$ with some $c \geq 0$, and $c > 0$ if and only if
$\p \{ B > 0 \} > 0$.

\begin{theorem} \label{th:max=1}
Assume (\ref{eq:A-ass}), $\E |B|^\nu < \infty$ for some $\nu > \kappa$. 
Then for the tail of the solution of (\ref{eq:maxeq}) we have
\begin{equation} \label{eq:maxtail}
\lim_{x \to \infty} m(\log x) x^\kappa \p \{ X > x \}  = C_\alpha
\frac{1}{\kappa} \E [ (AX_+ \vee B_+)^\kappa - (AX_+)^\kappa ].
\end{equation}
\end{theorem}

Equation (\ref{eq:maxeq}) has an important application in the analysis of the 
maximum of random walks with negative drift. Indeed, if $B \equiv 1$, then $\log X = M$, 
where
$M = \max \{0, S_1, S_2, \ldots \}$, and 
$S_n = \log A_1 + \log A_2 + \ldots + \log A_n$, where $\log A_1, \log A_2, \ldots$ are 
iid $\log A$. For general nonnegative $B$ the equation corresponds to the maximum of 
perturbed random walks; see Iksanov \cite{Iksanov2}.

Finally, we note that the tail behavior (\ref{eq:tail}) with nontrivial slowly 
varying function was noted before by Rivero \cite{Rivero} for exponential 
functionals of L\'evy processes. In Counterexample 1 \cite{Rivero} the 
following is shown. Let $(\sigma_t)_{t \geq 0}$ be a nonlattice subordinator, 
such that $\E e^{\kappa \sigma_1} < \infty$ and
$m(x) = \E I(\sigma_1 > x) e^{\kappa \sigma_1}$ is regularly varying with index 
$-\alpha \in (-1/2,-1)$. Consider the L\'evy process $(\xi_t)_{t \geq 0}$ 
obtained by killing $\sigma$ at $\zeta$, an independent exponential time with 
parameter $\log \E e^{\kappa \sigma_1}$. Then, in terms of the 
exponential functional $J=\int_0^\zeta e^{\xi_t} {\mathrm{d}} t$ Lemma 4 \cite{Rivero} 
states that 
$\lim_{x \to \infty} m(\log x) x^{\kappa} \p \{ J > x \} = c$,
for some $c > 0$.

\bigskip

Assume now that $\E A^\kappa = \theta < 1$ for some $\kappa > 0$, and
$\E A^t = \infty$ for any $t > \kappa$. Consider the new probability measure
\[ 
\p_\kappa \{ \log A \in C \} = \theta^{-1} \E  [I(\log A \in C) A^\kappa],
\] 
that is under the new measure $\log A$ has df
\[ 
F_\kappa (x) = \theta^{-1}  \int_{-\infty}^x e^{\kappa y} F({\mathrm{d}} y).
\] 
Note that these are the same definitions as in (\ref{eq:alphameasure}) and 
(\ref{eq:Fkappa}), the only difference is that now $\theta < 1$. Therefore the 
same notation should not be confusing.
The assumption $\E A^t = \infty$ for all $t > \kappa$ means that $F_\kappa$ is 
heavy-tailed. 
Rewriting again (\ref{eq:equation}) under the new measure $\p_\kappa$ leads now
to a \emph{defective} renewal equation for 
the tail of $X$. To analyze the asymptotic behavior of the resulting equation we use the 
techniques and results developed by Asmussen, Foss and Korshunov \cite{AFK}. A slight 
modification of their setup is necessary, since our df $F_\kappa$ is not concentrated on 
$[0,\infty)$.

For some $T \in (0,\infty]$ let $\Delta = (0, T]$. For a df $H$ we put
$H(x + \Delta) = H(x+T) - H(x)$.
A df $H$ on $\R$ is in the class $\mathcal{L}_\Delta$ if 
$H(x+t+\Delta) / H(x+\Delta) \to 1$ uniformly in $t \in [0,1]$, and
it belongs to the class of $\Delta$-subexponential 
distributions, $H \in \mathcal{S}_\Delta$, if
$H(x + \Delta) > 0$ for $x$ large enough, $H \in \mathcal{L}_\Delta$,
and $(H * H)(x + \Delta) \sim 2 H(x+\Delta)$.
If $H \in \mathcal{S}_\Delta$ for every $T > 0$, then it is called \emph{locally 
subexponential}, $H \in \mathcal{S}_{loc}$.
The definition of the class $\mathcal{S}_\Delta$ is given by Asmussen, Foss and 
Korshunov 
\cite{AFK} for distributions on $[0,\infty)$ and by Foss, Korshunov and Zachary
\cite[Section 4.7]{FKZ} for distributions on $\R$.
In order to use a slight extension of Theorem 5 \cite{AFK}
we need the additional natural assumption 
$\sup_{y > x} F_\kappa(y + \Delta) = O(F_\kappa(x+\Delta))$ for $x$ large 
enough. Properties of locally subexponential distributions, in particular its relation to 
infinitely divisible distributions were investigated by Watanabe and Yamamuro
\cite{WY1, WY2}. Our assumptions on $A$ are the following:
\begin{equation} \label{eq:A-ass<1}
\begin{gathered}
\E A^\kappa = \theta < 1, \  \kappa > 0, \quad
F_\kappa \in \mathcal{S}_{loc}, \quad 
\sup_{y > x} F_\kappa(y + \Delta) 
= O(F_\kappa(x+\Delta)) \text{ for $x$ large enough,}\\
\text{ and $\log A$ conditioned on $A$ being nonzero is nonarithmetic}. 
\end{gathered}
\end{equation}

\begin{theorem} \label{th:<1}
Assume (\ref{eq:A-ass<1}), $\E |B|^\nu < \infty$ for some 
$\nu > \kappa$, $\p \{ A x + B = x\} < 1$ for any $x \in \R$.
Then for the tail of the solution of the perpetuity equation 
(\ref{eq:equation}) we have
\begin{equation} \label{eq:<1asy}
\begin{split}
& \lim_{x \to \infty} g(\log x)^{-1} x^\kappa \p \{ X > x \} =
\frac{\theta}{(1-\theta)^2 \kappa} \E [ (AX + B)_+^\kappa - (AX)_+^\kappa ],\\
& \lim_{x \to \infty}g(\log x)^{-1} x^\kappa \p \{ X \leq -x \}  =
\frac{\theta}{(1-\theta)^2 \kappa} \E [ (AX + B)_-^\kappa - (AX)_-^\kappa ], 
\end{split}
\end{equation}
where $g(x) =  F_\kappa (x+1) - F_\kappa(x)$.
Moreover,
$\E [ (AX + B)_+^\kappa - (AX)_+^\kappa ] +
\E [ (AX + B)_-^\kappa - (AX)_-^\kappa ] > 0$.
\end{theorem}

Note that the condition $F_\kappa \in \mathcal{L}_\Delta$ with $\Delta=(0,1]$ implies 
that $g(\log x)$ is slowly varying. Indeed, for any $\lambda > 0$
\[
\frac{g(\log (\lambda x)) }{g(\log x)} = 
\frac{F_\kappa( \log x + \log \lambda + \Delta)}{F_\kappa( \log x + \Delta)} \to 1.
\]

The condition $F_\kappa \in \mathcal{S}_{loc}$ 
is much stronger than the corresponding regularly varying condition in Theorem 
\ref{th:=1}. Typical examples satisfying this condition are the Pareto, lognormal and 
Weibull (with parameter less than 1) distributions, see \cite[Section 4]{AFK}. For 
example in the Pareto case, i.e.~if for large enough $x$ we 
have $\overline F_\kappa(x) = c \, x^{-\beta}$ for some $c > 0, \beta > 0$, then
$g(x) \sim c \beta x^{-\beta - 1}$, and so
$\p \{ X > x \} \sim c' x^{-\kappa} (\log x)^{-\beta -1}$.
In the lognormal case, when $F_\kappa(x) = \Phi(\log x)$ for $x$ large enough,
with $\Phi$ being the standard normal df,
(\ref{eq:<1asy}) gives the asymptotic
$\p \{ X > x\} \sim c x^{-\kappa} e^{-(\log \log x)^2/2} / \log x$, $c > 0$.
Finally, for Weibull tails
$\overline F_\kappa(x) = e^{-x^\beta}$, $\beta \in (0,1)$, we obtain
$\p \{ X > x\} \sim c x^{-\kappa} (\log x)^{\beta-1} e^{-(\log x)^\beta}$,
$c > 0$.

\begin{theorem} \label{th:max<1}
Assume (\ref{eq:A-ass<1}), $\E |B|^\nu < \infty$ for some 
$\nu > \kappa$. Then for the tail of the solution of (\ref{eq:maxeq}) we have
\begin{equation} \label{eq:maxtail<1}
\lim_{x \to \infty} g(\log x)^{-1} x^\kappa \p \{ X > x \} =
\frac{\theta}{(1-\theta)^2 \kappa} \E [ (AX_+ \vee B_+)^\kappa - (AX_+)^\kappa ],
\end{equation}
where $g(x) =  F_\kappa (x+1) - F_\kappa(x)$.
\end{theorem}

In the special case $B \equiv 1$ we have the following.

\begin{corollary} \label{cor:max}
Let $\log A_1, \log A_2, \ldots$ be independent copies of $\log A$, let 
$S_n = \log A_1 + \log A_2 + \ldots + \log A_n$ denote their partial sum, and
$M = \max \{0, S_1, S_2, \ldots \}$.
Assume (\ref{eq:A-ass<1}). Then
\[ 
\p \{ M > x \} \sim c g(x) e^{-\kappa x},
\] 
with some $c > 0$, where $g(x) = F_\kappa (x+1) - F_\kappa(x)$.
\end{corollary}

From Goldie's unified method, it is clear that
the asymptotic behavior of the solution $X$ of (\ref{eq:equation}) is closely related to 
the maximum $M$ of the corresponding random walk $S_n$. 
The assumption $\E A^\kappa = 1$ together with $A \not\equiv 1$ imply 
that $\E \log A < 0$, so the 
random walk tends to $- \infty$, thus $M$ is $\p$-a.s.~finite. Assuming 
(\ref{eq:assump-2}) Korshunov \cite{Kor} showed for $\alpha > 1/2$
(all he needs is the SRT, so the same holds under (\ref{eq:Doney-cond}) for 
$\alpha \in (0,1)$) that for some constant $c > 0$
\[
\lim_{x \to \infty} \p \{ M > x \} e^{\kappa x} m(x)   = c.
\] 
Thus Theorem \ref{th:max=1} contains Korshunov's result 
\cite{Kor}. However, note that Korshunov obtained the corresponding liminf result 
in (\ref{eq:maxtail}), when the SRT does not hold. With our method the liminf result does 
not follow due to the smoothing transform (\ref{eq:def-smooth}). The problem is to 
`unsmooth' the liminf version of (\ref{eq:hat-f-asy}). The same difficulty appears in the 
perpetuity case.

In specific cases the connection between $M$ and $X$ are even more transparent. Let 
$(\xi_t)_{t \geq 0}$ be a 
nonmonotone L\'evy process, which tends to $-\infty$. Consider its exponential 
functional $J= \int_0^\infty e^{\xi_t} {\mathrm{d}} t$ and its supremum
$\overline \xi_\infty = \sup_{t \geq 0} \xi_t$. Arista and Rivero \cite[Theorem 4]{AR} 
showed that $\p \{ J > x\}$ is regularly varying with parameter $-\alpha$ 
if and only if $\p \{ e^{\overline \xi_\infty} > x\}$ is regularly varying with the same 
parameter. Now, if $(\xi_t)$ has finite jump activity and 0 drift, then conditioning on 
its first jump time one has the perpetuity equation
\[
J \stackrel{\mathcal{D}}{=} A J + B, 
\]
with $B$ being an exponential random variable independent of $A$, and 
$\log A$ is the jump  size.
From this we see that in this special case, Theorem \ref{th:=1}/\ref{th:<1} and  
Theorem \ref{th:max=1}/\ref{th:max<1} with $B \equiv 1$ are equivalent.

Finally, we note that using Alsmeyer's sandwich method \cite{Alsmeyer} it is possible to 
apply our results to iterated function systems.

\section{Proofs}

First, we prove the analogue of Goldie's implicit renewal theorem \cite[Theorem 
2.3]{Goldie} in both cases.

\begin{theorem} \label{th:impren}
Assume (\ref{eq:A-ass}), and for some random variable $X$
\[
\begin{split}
\int_0^\infty | \p \{ X > x\} - \p \{ AX > x\} | x^{\kappa + \delta -1} {\mathrm{d}} x < 
\infty 
\end{split}
\]
for some $\delta > 0$, where $X$ and $A$ are independent. Then
\[
\lim_{x \to \infty} m(\log x) x^\kappa \p \{ X > x \} = C_\alpha  
\int_0^\infty [ \p \{ X > x\} - \p \{ AX > x\} ] x^{\kappa -1} {\mathrm{d}} x.
\]
\end{theorem}

\begin{proof}
We follow closely Goldie's proof. Put
\begin{equation} \label{eq:def-psi}
\psi(x) = e^{\kappa x} (\p \{ X > e^x \} - \p \{ AX > e^x \}), \ 
f(x) = e^{\kappa x} \p \{ X > e^x \}.
\end{equation}
Using that $X$ and $A$ are independent we obtain the equation
\begin{equation} \label{eq:ren0}
f(x) = \psi(x) + \E f(x - \log A) A^\kappa. 
\end{equation}
Using the definition (\ref{eq:alphameasure}) we see that
$\E_\kappa g(\log A) = \E (g(\log A) A^\kappa)$ thus under the new measure equation 
(\ref{eq:ren0}) reads as
\begin{equation} \label{eq:renewal}
f(x) = \psi(x) + \E_\kappa f(x - \log A).
\end{equation}
We have to introduce a smoothing transformation, since the function
$\psi$ is not necessarily directly Riemann integrable (dRi), which is needed to 
apply the key renewal theorem.
Introduce the smoothing transform of $g$ as
\begin{equation} \label{eq:def-smooth}
\hat g(s) = \int_{-\infty}^s e^{-(s-x)} g(x) {\mathrm{d}} x. 
\end{equation}
Applying this transform to both sides of (\ref{eq:renewal})
we get the renewal equation
\begin{equation} \label{eq:ren-smooth}
\hat f(s) = \hat \psi(s) + \E_\kappa \hat f(s - \log A).
\end{equation}
Iterating (\ref{eq:ren-smooth}) we obtain for any $n \geq 1$
\begin{equation} \label{eq:iter}
\hat f(s) = \sum_{k=0}^{n-1} \int_\R \hat \psi(s - y ) F_\kappa^{*k}({\mathrm{d}} y)
+ \E_\kappa \hat f(s - S_n),
\end{equation}
where $\log A_1, \log A_2, \ldots $ are independent copies of $\log A$ (under $\p$ and 
$\p_\kappa$), independent of $X$, and 
$S_n = \log A_1 + \ldots + \log A_n$. Since $S_n \to -\infty$ $\p$-a.s.
\[
\E_\kappa \hat f(s - S_n) = e^{-s} \int_{-\infty}^s e^{(\kappa + 1) y}
\p \{ X e^{S_n} > e^y \} {\mathrm{d}} y \to 0 \quad \text{as } n \to \infty,
\]
where we also used that $\E_\kappa g(S_n) = \E g(S_n) e^{\kappa S_n}$.
Therefore as $n \to \infty$ from (\ref{eq:iter}) we have
\begin{equation} \label{eq:smoothrenew}
\hat f(s) = \int_\R \hat \psi(s-y) U({\mathrm{d}} y), 
\end{equation}
where $U(x) = \sum_{n=0}^\infty F_\kappa^{*n}(x)$ is the renewal function of 
$F_\kappa$. The question is under what conditions of $z$ the key 
renewal theorem
\begin{equation} \label{eq:int-asy}
m(x) \int_\R z(x-y) U({\mathrm{d}} y) \to C_\alpha \int_\R z(y) {\mathrm{d}} y
\end{equation}
holds.
In the following lemma, which is a modification of Erickson's Theorem 3 
\cite{Eri70}, we give sufficient condition for $z$ to
(\ref{eq:int-asy}) hold. We note that both in Lemma 3 \cite{Rivero} and in Theorem 4 of 
\cite{VT} the authors wrongly claim that (\ref{eq:int-asy}) holds if $z$ is
dRi. A counterexample is given in the Appendix.
The same statement under less restrictive condition is shown using stopping time argument 
in \cite[Proposition 6.4.2]{Iksanov}.  
For the sake of completeness we give a proof here.

\begin{lemma} \label{lemma:keyren}
Assume that $z$ is dRi, $z(x) = O(x^{-1})$ as $x \to \infty$, and (\ref{eq:A-ass}) holds.
Then (\ref{eq:renewal-alpha}) implies (\ref{eq:int-asy}).
\end{lemma}

\begin{proof}
Using the decomposition $z= z_+ - z_-$ we may and do assume that $z$ is nonnegative.
Write
\[
\begin{split}
& m(x) \int_\R z(x-y) U({\mathrm{d}} y) \\
& = m(x) \left[ \int_x^\infty z(x-y) U({\mathrm{d}} y) +
\int_0^x z(x-y) U({\mathrm{d}} y) +  \int_{-\infty}^0 z(x-y) U({\mathrm{d}} y) \right]\\
&=:I_1(x) + I_2(x) + I_3(x).
\end{split}
\]

We first show that $I_1(x) \to C_\alpha \int_{-\infty}^0 z(y) {\mathrm{d}} y$ whenever $z$ 
is dRi.
Fix $h > 0$ and put $z_k(x) = I(x \in  [(k-1)h, kh))$, 
$a_k = \inf \{z(x) : x \in  [(k-1)h, kh) \}$, and
$b_k = \sup \{z(x) : x \in  [(k-1)h, kh) \}$, $k \in \Z$.
Simply
\[
m(x) \sum_{k=-\infty}^{0} a_k (U * z_k)(x) \leq I_1(x) \leq 
m(x) \sum_{k=-\infty}^{0} b_k (U * z_k)(x) .
\]
As $x \to \infty$ by (\ref{eq:renewal-alpha}) for any fixed $k$
\[
\begin{split}
m(x)  (U * z_k)(x)
& = \frac{m(x)}{m(x-kh)} m(x-kh) [ U(x - kh + h) - U(x - k h) ] 
\to C_\alpha h,
\end{split}
\]
where the convergence $m(x)/m(x-kh) \to 1$
follows from the fact that $m$ is regularly varying with index $1 - \alpha$.
Since $m$ is nondecreasing and $k \leq 0$ this also gives us an integrable majorant
uniformly in $k \leq 0$, i.e.~for $x$ large enough
$\sup_{k < 0}  m(x)  (U * z_k)(x) \leq 2 C_\alpha h$.
Thus by Lebesgue's dominated convergence theorem
\[
\lim_{x \to \infty} m(x) \sum_{k=-\infty}^{0} a_k (U * z_k)(x) =
C_\alpha \sum_{k= -\infty}^0 a_k  h, 
\]
and similarly for the upper bound.
Since $z$ is dRi the statement follows.

The convergence $I_2 (x) \to C_\alpha \int_0^\infty z(x) {\mathrm{d}} x$
follows exactly as in the proof of \cite[Theorem 3]{Eri70}, since in that proof 
only formula (\ref{eq:renewal-alpha}) and its consequence $U(x) \sim C_\alpha x / 
(\alpha m(x))$ are used.

Finally, we show that $I_3(x) \to 0$. Indeed,
\[
\begin{split}
m(x) \int_{-\infty}^{0} z(x-y) U({\mathrm{d}} y)  
\leq K m(x)  \int_{-\infty}^{0} (x-y)^{-1} U({\mathrm{d}} y) 
\leq K \frac{m(x)}{x} U(0) \to 0,
\end{split}
\]
where $K \geq \sup_{x > 0} x z(x)$.
\end{proof}

Recall the definition of $\psi$ in (\ref{eq:def-psi}).
Next we show that
$\hat \psi$ satisfies the condition of Lemma \ref{lemma:keyren}.
Indeed,
\begin{equation} \label{eq:psi-est}
\begin{split}
\hat \psi(s) & = e^{-s} \int_{-\infty}^s e^{(\kappa+1) x} 
[\p \{ X > e^x\} - \p \{ AX > e^x \} ] {\mathrm{d}} x  \\
& \leq e^{-s} \int_0^{e^s} y^\kappa | \p \{ X > y \} - \p \{ AX > y \} | {\mathrm{d}} y \\
& \leq e^{- \delta s} \int_0^\infty y^{\kappa + \delta - 1}
| \p \{ X > y \} - \p \{ AX > y \} | {\mathrm{d}} y,
\end{split}
\end{equation}
and the last integral is finite due to our assumptions. The same calculation shows that
\[
\int_\R \hat \psi(s) {\mathrm{d}} s = \int_\R \psi(x) {\mathrm{d}} x = 
\int_0^\infty y^{\kappa - 1} [\p \{ X > y \} - \p \{ AX > y \} ] {\mathrm{d}} y.
\]
It follows from \cite[Lemma 9.2]{Goldie} that $\hat \psi$ is dRi, thus
from Lemma \ref{lemma:keyren} and (\ref{eq:psi-est}) we obtain that
for the solution of (\ref{eq:smoothrenew})
\begin{equation} \label{eq:hat-f-asy}
\lim_{s \to \infty} m(s) \hat f(s) = C_\alpha \int_\R \psi(y) {\mathrm{d}} y =: c.
\end{equation}
From (\ref{eq:hat-f-asy}) the statement follows again in the same way as in
\cite[Lemma 9.3]{Goldie}. Indeed, since $m(x)$ is regularly varying $m(\log x)$ is slowly 
varying, thus from  (\ref{eq:hat-f-asy}) we obtain for any $0 < a < 1 < b < \infty$
\[
m(\log x) \frac{1}{x} \int_{ax}^{bx} y^\kappa \p \{ X > y \} {\mathrm{d}} y \to (b-a) c.
\]
With $a = 1 < b$ and $a < 1 = b$ it follows that
\[
\begin{split}
c \frac{b-1}{b^{\kappa + 1} - 1} (\kappa + 1) & \leq
\liminf_{x \to \infty} x^\kappa m(\log x) \p \{ X > x \} \\
&\leq \limsup_{x \to \infty} x^\kappa m(\log x) \p \{ X > x \} 
\leq c \frac{1-a}{1- a^{\kappa + 1} } (\kappa + 1).
\end{split}
\]
As $a \uparrow 1, b \downarrow 1$ the convergence follows.
\end{proof}

\begin{theorem} \label{th:impren<1}
Assume (\ref{eq:A-ass<1}), and for some random variable $X$
\[
\begin{split}
\int_0^\infty | \p \{ X > x\} - \p \{ AX > x\} | x^{\kappa + \delta -1} {\mathrm{d}} x < 
\infty 
\end{split}
\]
for some $\delta > 0$, where $X$ and $A$ are independent. Then
\[
\lim_{x \to \infty} g(\log x)^{-1} x^\kappa \p \{ X > x \} = 
\frac{\theta}{(1-\theta)^2}
\int_0^\infty [ \p \{ X > x\} - \p \{ AX > x\} ] x^{\kappa -1} {\mathrm{d}} x.
\]
\end{theorem}

\begin{proof}
Following the same steps as in the proof of Theorem \ref{th:impren}
we obtain
\[
\hat f(s) = \int_\R \hat \psi(s-y) U({\mathrm{d}} y), 
\]
where $U$ is the defective renewal function 
$U(x) = \sum_{n=0}^\infty (\theta F_\kappa)^{*n}(x)$. Since $\theta < 1$ we have
$U(\R) = (1 -\theta)^{-1} < \infty$. A modification of Theorem 5 \cite{AFK} gives the 
following. Recall $g$ from Theorem \ref{th:<1}.

\begin{lemma} \label{lemma:ren<1}
Assume (\ref{eq:A-ass<1}), $z$ is dRi, and $z(x) = o(g(x))$. Then
\[
\int_\R z(x-y) U({\mathrm{d}} y) \sim \theta g(x) \frac{\int_\R z(y) {\mathrm{d}} y}{(1 - 
\theta)^2}. 
\]
\end{lemma}

\begin{proof}
By the decomposition $z= z_+ - z_-$, we may and do assume that $z$ is nonnegative.
We again split the integral 
\[
\int_\R z(x-y) U({\mathrm{d}} y) = I_1(x) + I_2(x) + I_3(x),
\]
where $I_1, I_2$, and $I_3$ are the integrals on $(x,\infty)$, $(0,x]$, and on 
$(-\infty, 0]$, respectively.

The asymptotics $I_1(x) \sim \theta g(x) \int_{-\infty}^0 z(y) {\mathrm{d}} y / 
(1-\theta)^2$ 
follows along the same 
lines as in the proof of Lemma \ref{lemma:keyren}. Theorem 5(i) \cite{AFK} gives
$I_2(x) \sim \theta g(x) \int_0^\infty z(y) {\mathrm{d}} y / (1 - \theta)^2$. (In the 
Appendix we 
explain 
why the results for $\Delta$-subexponential distributions on $[0,\infty)$  
remain true in 
our case.) Finally, for $I_3$ we have
\[
I_3(x) \leq U(0) \, \sup_{y \geq x} z(y) = o(g(x)),
\]
where we use that $\sup_{y \geq x} F_\kappa(y + \Delta) = O(F_\kappa(x+\Delta))$.
\end{proof}

As in (\ref{eq:psi-est}) we have $\hat \psi(x) = O(e^{-\delta x})$ for some 
$\delta > 0$. Since $F_\kappa$ is subexponential
$\hat \psi(x) = o(g(x))$. That is, the condition of Lemma \ref{lemma:ren<1} 
holds, and we obtain the asymptotic
\[
\hat f(s) \sim  \theta g(s) \frac{\int_\R \psi(y) {\mathrm{d}} y}{(1 - \theta)^2}, \quad s 
\to 
\infty.
\]
Since $g(x)$ is subexponential, $g(\log x)$ is slowly varying, and the
proof can be finished  in exactly the same way as in Theorem \ref{th:impren}.
\end{proof}

Now the proofs of Theorems \ref{th:=1}, \ref{th:<1}, \ref{th:max=1}, \ref{th:max<1} are 
applications of the corresponding implicit renewal theorem.

\begin{proof}[Proofs of Theorems \ref{th:max=1} and \ref{th:max<1}]
The existence of the unique solution of (\ref{eq:maxeq}) follows from 
\cite[Proposition 5.1]{Goldie}. Choose $\delta \in (0,\nu - \kappa)$.
Since $| \p \{ AX \vee B > x\} - \p \{ AX > x\} | = \p \{ AX \vee B > x \geq AX \}$,
Fubini's theorem gives
\[
\begin{split}
& \int_0^\infty | \p \{ AX \vee B > x\} - \p \{ AX > x\} | x^{\kappa + \delta -1} 
{\mathrm{d}} x
= \int_0^\infty \p \{ AX \vee B > x \geq AX \} x^{\kappa + \delta -1} {\mathrm{d}} x  \\
& = (\kappa + \delta)^{-1} \E [ (AX \vee B)_+^{\kappa+\delta} - (AX)_+^{\kappa+\delta} ]
\leq (\kappa + \delta)^{-1} \E B_+^{\kappa + \delta}.
\end{split}
\]
Therefore (\ref{eq:maxtail}) and (\ref{eq:maxtail<1}) follows from Theorem 
\ref{th:impren} and \ref{th:impren<1}, respectively.
The form of the limit constant follows similarly. Note that for $B \equiv 1$, i.e.~when 
$\log X = M$ is the maximum of a random walk with negative drift, the constant is 
strictly positive.
\end{proof}

\begin{proof}[Proofs of Theorems \ref{th:=1} and \ref{th:<1}]
The existence of the unique solution of (\ref{eq:equation}) is well-known.
Let us choose $\delta > 0$ so small that
\begin{equation} \label{eq:eps-cond}
\kappa + \frac{3 \kappa \delta}{1-\delta} < \nu \text{ for } \kappa \geq 1,
\text{ and } \kappa + \delta \leq \min\{1, \nu\} \text{ for } \kappa < 1.
\end{equation}
Note that
\[
|\p \{ AX + B > y \} - \p \{ AX > y \} | \leq \p \{ AX + B > y \geq AX \}
+ \p \{ AX > y \geq AX + B \}.
\]
Now Fubini's theorem gives  for the first term
\[
\int_0^{\infty} \p \{ AX + B > y \geq AX \} y^{\kappa - 1 + \delta} {\mathrm{d}} y 
\leq (\kappa + \delta)^{-1}
\E I(B \geq 0) ((AX + B)_+^{\kappa + \delta} - (AX)_+^{\kappa + \delta}).
\]
The same calculation for the second term implies
\[
\int_0^\infty |\p \{ AX + B > y \} - \p \{ AX > y \} | y^{\kappa + \delta -1} {\mathrm{d}} 
y 
\leq (\kappa + \delta)^{-1}
\E  | (AX + B)_+^{\kappa + \delta} - (AX)_+^{\kappa + \delta} | .
\]
We show that the expectation on the right-hand side is finite. Indeed, for
$a, b \in \R$ we have
$|(a + b )_+^\gamma - a_+^\gamma| \leq |b|^\gamma$ for $\gamma \leq 1$ and
$|(a + b)_+^\gamma - a_+^\gamma| \leq 2 \gamma |b| (|a|^{\gamma -1} + |b|^{\gamma-1})$
for $\gamma > 1$. 
From Theorem 1.4 by Alsmeyer, Iksanov and R\"osler \cite{AIR}
we know that $\E |X|^\gamma  < \infty$ for any
$\gamma < \kappa$. (We note that for $\kappa > 1$ this also follows from Theorem 5.1 by
Vervaat \cite{Vervaat}. Actually, \cite[Theorem 1.4]{AIR} states equivalence.)
Assume that $\kappa \geq 1$ and let
$p = \kappa + 2\kappa \delta / (1 - \delta)$, $1/q = 1 - 1/p$. By H\"older's 
inequality and by the choice of $\delta$ in (\ref{eq:eps-cond})
\[
\begin{split}
& \E  | (AX + B)_+^{\kappa + \delta} - (AX)_+^{\kappa + \delta} |
\leq 2 (\kappa + \delta) \left[
\E |B| |AX|^{\kappa + \delta - 1} 
+ \E |B|^{\kappa + \delta} \right] \\
&\leq 2 (\kappa + \delta) \left[
\E |X|^{\kappa +\delta -1} 
(\E |B|^p)^{1/p} (\E A^{q(\kappa + \delta - 1)} )^{1/q}
+ \E |B|^{\kappa + \delta} \right] < \infty,
\end{split}
\]
which proves the statement for $\kappa \geq 1$. For $\kappa < 1$ we choose $\delta$ 
such that $\kappa + \delta \leq 1$, so
\[
\E  | |AX + B|^{\kappa + \delta} - |AX|^{\kappa + \delta} |
\leq \E |B|^{\kappa + \delta} < \infty.
\]

Finally, the positivity of the limit follows in exactly the same way as in \cite{Goldie}.
Goldie shows \cite[p.157]{Goldie} that for some positive constants $c, C > 0$
\[ 
\p \{ |X| > x \} \geq c \, \p \left\{ \max \{0, S_1, S_2, \ldots\} > C + \log x \right\}.
\] 
Now the positivity follows from Theorem \ref{th:max=1} and \ref{th:max<1}, respectively,
with $B \equiv 1$.
\end{proof}

Note that for Lemma \ref{lemma:keyren} we only need that 
$\hat \psi(s) = O(s^{-1})$. Obvious 
modification of the proof shows that this holds if
instead of assuming $\E |B|^\nu < \infty$ for some $\nu > \kappa$ we assume that
\[ 
\E |B|^\kappa (\log_+ |B|)^{\max\{ 1, \kappa\}} < \infty, \ 
\E |B|^\kappa \log_+ A < \infty, \text{ and } \
\E |B| A^{\kappa-1} \log_+ A < \infty.
\]
Clearly, the latter condition is weaker for independent $A$ and $B$.

\section{Appendix}

\subsection{Strong renewal theorem}

In this subsection we show that a slight extension of the strong renewal theorem by 
Caravenna \cite{Car} and Doney \cite{Doney} holds. The proof follows along the same lines 
as the proof of Caravenna \cite{Car}, so we only sketch the proof and emphasize the 
differences. For convenience, we also use Caravenna's notation.

\begin{theorem} \label{th:SRT}
Assume that the distribution function $H$ is nonarithmetic, and
for some $c, \kappa > 0$, $\alpha \in (0,1)$, and for a slowly varying 
function $\ell$ we have
\begin{equation} \label{eq:app-ass}
H(-x) \leq  c e^{-\kappa x}, \quad 1 - H(x) = \overline H(x) = \frac{\ell(x)}{x^\alpha},
\quad x > 0.
\end{equation}
Then, for the renewal function $U(x) =  \sum_{n=0}^\infty H^{*n}(x)$
\begin{equation} \label{eq:app-SRT}
\lim_{x \to \infty} m(x) [ U(x+h) - U(x)] = h C_\alpha
\end{equation}
holds for any $h> 0$ with $m(x) = \int_0^x \overline H(u) {\mathrm{d}} u$, if and only if
\begin{equation} \label{eq:app-iff}
\lim_{\delta \to 0} \limsup_{x \to \infty} x \overline H(x) 
\int_1^{\delta x} \frac{1}{y \overline H(y)^2} H(x - {\mathrm{d}} y) = 0. 
\end{equation}
\end{theorem}

In the following, $X, X_1, X_2, \ldots$ are iid random variables with df $H$,  
$S_n = X_1 + \ldots + X_n$ is their partial sum.
We may choose a strictly increasing differentiable function $A$, such that 
$\overline H(x) \sim 1/A(x)$, and $A'(x) \sim \alpha A(x) / x$
(see p.~15 \cite{BGT}). The norming constant  $a_n$, for which $S_n / a_n$ 
converges in distribution to an $\alpha$-stable law, can be given as
$a_n = a(n) = A^{-1}(n)$, where $A^{-1}$ is the inverse of $A$.
We fix an $h > 0$, and put $I= (-h,0]$.

\smallskip

First note that Lemmas 4.1 and 4.2 in \cite{Car} hold true for every $H$ in the 
domain of attraction of an $\alpha$-stable law; nonnegativity is not needed here. We show 
that the simple Lemma 5.1 in \cite{Car} also remains true in our case. This is needed to 
treat the small values of $z$.

\begin{lemma}
Assume that (\ref{eq:app-ass}) holds. Then there exist $C, C' > 0$, such that for all
$n \in \N$ and $z > 0$
\[
\p \{ S_n \in z + I \} \leq \frac{C}{a_n} e^{-C' n / A(z) }.
\]
\end{lemma}

\begin{proof}
We have to handle the case $\p\{ S_n \leq 0\}$. 
A 
truncation argument combined with Chernoff's bounding technique shows that for 
any fixed $K > 0$ there is $c, c' > 0$, such that for all $n \in \N$
\begin{equation} \label{eq:app-expbound}
\p  \{ S_n \leq K \} \leq c e^{-c' n}.
\end{equation}
Thus for $z \geq 0$, with $M_n = \max\{ X_i \, : \, i =1,2,\ldots, n\}$,
\[
\begin{split}
\p \{ S_n \leq z\} & \leq   \p \{ M_n \leq z \} + \p \{ S_n \leq z, M_n > z \} \\
& \leq c_1 e^{-c_2 n \overline H(z)} + n \int_z^\infty \p \{ S_{n-1} \leq z- y\} 
H({\mathrm{d}} y) 
\\
& \leq c_3 e^{- c_4 n \overline H(z)}.
\end{split}
\]
Assuming that $n$ is even, using Stone's local limit theorem \cite[Theorem 8.4.2]{BGT} we 
obtain
\[
\begin{split}
&\p  \{ S_n \in z+I \}  =
\int_0^z \p \{ S_{n/2} \in z - y +I \} \p \{ S_{n/2} \in {\mathrm{d}} y \}
+ \int_z^\infty \p \{ S_{n/2} \in z - y +I \} \p \{ S_{n/2} \in {\mathrm{d}} y \} \\
& \leq c_5  a_n^{-1} \p \{ S_{n/2} \leq z\} +
\sum_{k=0}^\infty \p \{ S_{n/2} \in [z + kh, z+ (k+1) h]\} 
\p \{ S_{n/2} \in [- (k+2)h, -k h]\} \\
& \leq c_3 c_5  a_n^{-1}  e^{- c_4 \frac{n}{2} \overline H(z)} + c_6 a_n^{-1} e^{-c_7 n},
\end{split}
\]
and the statement is proven.
\end{proof}

\begin{proof}[Proof of Theorem \ref{th:SRT}]
The necessity is shown in Theorem 1.9 by Caravenna \cite{Car} (see also 
\cite[Lemma 3.1]{Car} and note that in our case $q = 0$).

It is known (in the general two-sided case) that (\ref{eq:app-SRT}) is equivalent to
\[ 
\lim_{\delta \downarrow 0} \limsup_{x \to \infty} \frac{x}{A(x)}
\sum_{n=1}^{A(\delta x)} \p \{ S_n \in x + I \} = 0;
\] 
see (15) in \cite{Doney}, (3.4) in \cite{Car}, or the Appendix in \cite{Chi}.

As in \cite[Theorem 6.1]{Car} we claim that for $\ell \geq 0$
\begin{equation} \label{eq:app-srtcond2}
\lim_{\delta \downarrow 0} \limsup_{x \to \infty} \frac{x}{A(x)^{\ell + 1}}
\sum_{n=1}^{A(\delta x)} n^\ell \p \{ S_n \in x + I \} = 0.
\end{equation}
Using \cite[Lemma 4.1]{Car} we see that (\ref{eq:app-srtcond2}) holds for 
$\ell \geq \kappa_\alpha + 1 = \lfloor 1/\alpha \rfloor$, with $\lfloor \cdot \rfloor$ 
standing for the integer part. The proof goes by backward induction, i.e.~we assume that 
(\ref{eq:app-srtcond2}) holds for $\ell \geq \overline \ell + 1$, and we show it for
$\ell = \overline \ell$.

Let $B_n^m(y)$ denote the set (i.e.~a subset of the underlying probability space) on 
which there are exactly $m$ variables greater than $y$ among $X_1, \ldots, X_n$;
here $m \geq 0$, $y > 0$. Arguing as in the proof of \cite[Theorem 6.1]{Car}, we have to 
show that for $m \in \{ 1, 2, \ldots, \kappa_\alpha - \overline \ell\}$
\[ 
\lim_{\delta \downarrow 0} \limsup_{x \to \infty} \frac{x}{A(x)^{\overline \ell + 1}}
\sum_{n=1}^{A(\delta x)} n^{\overline \ell} \p \{ S_n \in x + I, B_n^m(\xi_{n,x}) \} = 0,
\] 
where $\xi_{n,x}= a_n^{\gamma_\alpha} x^{1 - \gamma_\alpha}$, with
$\gamma_\alpha = \alpha (1 - (1/\alpha - \lfloor 1/\alpha \rfloor)) / 4$. 
We have
\[
\begin{split}
&\p \{ S_n \in x + I, B_n^m(\xi_{n,x}) \} \leq 
n^m \p \{ S_n \in x + I, \min_{1 \leq i \leq m} X_i > \xi_{n,x},
\max_{m+1 \leq i \leq n} X_i \leq \xi_{n,x} \} \\
&\leq n^m \dint_{(y,w) \in (0,x]^2}  I(w \geq \xi_{n,x})
\p \{ S_{n-m} \in x -y + I, B_{n-m}^0(\xi_{n,x})  \} 
\p \{ S_m \in {\mathrm{d}} y, \min_{1 \leq i \leq m } X_i \in {\mathrm{d}} w \} \\
& \phantom{\leq} \  + n^m \int_x^\infty \p \{ S_{n-m} \in x -y + I  \} 
\p \{ S_m \in {\mathrm{d}} y \}.
\end{split}
\]
The first term is exactly the same as in the proof of \cite[Theorem 6.1]{Car}. More 
importantly, it also can be handled in exactly the same way, since only Lemmas 4.2 and 
5.1 in \cite{Car} are used, and the manipulations with the measure $\p \{ S_m \in 
{\mathrm{d}} y\}$ 
always concern the set where the minimum is positive. Therefore, to finish the proof it 
is enough to show that
\begin{equation} \label{eq:app-srtcond4}
\limsup_{x \to \infty} \frac{x}{A(x)^{\overline \ell + 1}}
\sum_{n=1}^{\infty} n^{\overline \ell + m} \int_x^\infty \p \{ S_{n-m} \in x -y + I  \} 
\p \{ S_m \in {\mathrm{d}} y \} = 0. 
\end{equation}
This is easy, since 
\begin{equation} \label{eq:app-in}
\begin{split}
&\int_x^\infty \p \{ S_{n-m} \in x -y + I  \} \p \{ S_m \in {\mathrm{d}} y \} \\
& \leq \sum_{k=0}^\infty \p \{ S_m \in [z + kh, z+ (k+1) h]\} 
\p \{ S_{n-m} \in [- (k+2)h, -k h]\} \\
& \leq c_1 \max_{y > x} \p \{ S_m \in y + I\} \p \{ S_{n-m} \leq 0 \} \\
& = o ( A(x)/x ) e^{-c_2 n},
\end{split}
\end{equation}
where at the last inequality we used (\ref{eq:app-expbound}) and the fact that 
(\ref{eq:app-iff}) implies that for any $m \geq 1$
\[
\lim_{x \to \infty} \frac{x}{A(x)} \p \{ S_m \in x+I\} = 0; 
\]
see Remark 3.2 in \cite{Car} or Remark 5 in \cite{Doney}.
Clearly, (\ref{eq:app-in}) implies (\ref{eq:app-srtcond4}), and the proof is finished.
\end{proof}

\subsection{A counterexample}

Here we give a counterexample to \cite[Lemma 3]{Rivero} and \cite[Theorem 4]{VT},
which shows that alone from the direct Riemann integrability of $z$  the key renewal 
theorem (\ref{eq:int-asy}) does not follow.

Let $a_n = n^{-\beta}$, with some $\beta > 1$, and let $d_n \uparrow \infty$ a sequence 
of integers. Consider the function $z$ that satisfies $z(d_n) = a_n$, 
$z(d_n \pm 1/2) = 0$, is linearly interpolated on the intervals $[d_n - 1/2, d_n + 1/2]$, 
and 0 otherwise. Since $\sum_{n=1}^\infty a_n < \infty$ the function $z$ is directly 
Riemann integrable.

Consider a renewal measure $U$ for which SRT (\ref{eq:renewal-alpha}) holds. Let $a > 
0$ be 
such that $U(a+1/4) - U(a-1/4) > 0$. 
From the proof of \cite[Theorem 3]{Eri70} it is clear that for any $\nu \in (0,1)$
\[
m(x) \int_{\nu x}^x z(x-y)  U( {\mathrm{d}} y) \to C_\alpha \int_0^\infty z(y) 
{\mathrm{d}} y.
\]
On the other hand for $x = a+ d_n$
\[
\begin{split}
\int_{a - 1/4}^{a+1/4} z(x-y) U( {\mathrm{d}} y )
& \geq \frac{a_n}{2} [U(a+1/4) - U(a-1/4)]
\end{split}
\]
Choosing $d_n = n^2$ and $\beta$ such that $2 \alpha + \beta < 2$, and recalling that $m$ 
is regularly varying with index $1-\alpha$, we see that 
$m(a+d_n) a_n \to \infty$, so the asymptotic (\ref{eq:int-asy})  cannot hold.

This example also shows that a growth condition on $z$, something like Erickson's
$z(x) = O(1/x)$, is needed.

\subsection{Local subexponentiality}

We claim that Theorem 5 in \cite{AFK} remains true in our setup.
Additionally to the local subexponential property, we assume that 
$\sup_{y \geq x} H(y + \Delta) = O(H(x + \Delta))$. The main technical tool in 
\cite{AFK} is the equivalence in Proposition 2. In our setup it has the 
following form.

\begin{lemma} \label{lemma:prop2}
Assume that $H \in \mathcal{L}_\Delta$, and 
$\sup_{y \geq x} H(y + \Delta) = O(H(x + \Delta))$. 
Let $X, Y$ be iid $H$. The following are 
equivalent:
\begin{itemize}
\item[(i)] $H \in \mathcal{S}_\Delta$;
\item[(ii)] there is a function $h$ such that $h(x) \to \infty$, $h(x) < x/2$, 
$H(x-y + \Delta) \sim H(x + \Delta)$ uniformly in $|y| \leq h(x)$, and
\[
\p \{ X + Y \in x + \Delta, X > h(x), Y > h(x) \} =
o(H(x + \Delta)).
\]
\end{itemize}
\end{lemma}

The proof is similar to the proof of Proposition 2 in \cite{AFK}, so it is omitted.
Assuming the extra growth condition all the results in \cite{AFK} 
hold true with the obvious modification of the proof.

\smallskip

Alternatively, we could use Theorem 1.1 by Watanabe and Yamamuro \cite{WY1}, 
from which the extension of Theorem 5 in \cite{AFK} follows. In Theorem 1.1
\cite{WY1} finite exponential moment for the left-tail is assumed, which is 
satisfied in our setup by (\ref{eq:Fknegtail}). Also note that Theorem 1.1 is 
a much stronger result what we need: it gives an equivalence of certain tails, 
and we only need implication (ii) $\Rightarrow$ (iii).

\medskip

\noindent
\textbf{Acknowledgement.} I thank M\'aty\'as Barczy for useful comments on the 
manuscript, V\'ictor Rivero for drawing my attention to Counterexample 1 
in \cite{Rivero}, and Alexander Iksanov for reference \cite{Iksanov}. 
I am thankful to the referees for the thorough reading of the paper, and for their 
comments and suggestions, which greatly improved the paper.
This research was funded by a postdoctoral fellowship
of the Alexander von Humboldt Foundation.

\end{document}